\documentclass[conference]{IEEEtran}
\usepackage{times}

\usepackage[numbers,sort&compress,square]{natbib}
\usepackage{multicol}

\usepackage{bm}
\usepackage{amsthm,array}
\usepackage{amsfonts,latexsym}
\usepackage{subfig,wrapfig}
\usepackage{graphicx}
\usepackage{times}
\usepackage{psfrag,epsfig}
\usepackage{verbatim}
\usepackage{tabularx}
\usepackage{color}
\usepackage{soul}
\usepackage{indentfirst}
\usepackage{inputenc}
\usepackage{courier}
\usepackage{mathtools}
\usepackage{algorithm,algpseudocode}
\usepackage{graphicx}
\usepackage{caption}
\usepackage[capitalise]{cleveref}
\usepackage{multirow}
\usepackage{tabularx}
\usepackage{color}

\def\0{\boldsymbol{0}}

\def\u{\boldsymbol{u}}

\def\v{\boldsymbol{v}}

\def\T{\mathcal{T}}

\def\I{\text{I}}
\def\d{\text{d}}

\def\R{\mathbb{R}}

\long\def\answer#1{}

\long\def\comment#1{}

\newcommand{\grad}[2]{D_{#1}{#2}}

\theoremstyle{definition}
\newtheorem{definition}{Definition}
\newtheorem{prop}{Proposition}
\newtheorem{lemma}{Lemma}
\newtheorem{corollary}{Corollary}

\def\ad{\text{ad}}

\newcommand{\fb}[2]{{#1}\Big({#2}\Big)}

\begin{document}
\title{Online Feedback Control for Input-Saturated Robotic Systems on Lie Groups}


\author{\IEEEauthorblockN{Taosha Fan and Todd Murphey}
\IEEEauthorblockA{Department of Mechanical Engineering, Northwestern University\\Email: taosha.fan@u.northwestern.edu, t-murphey@northwestern.edu}}


%

\maketitle
\begin{abstract}
In this paper, we propose an approach to designing online feedback controllers for input-saturated robotic systems evolving on Lie groups by extending the recently developed Sequential Action Control (SAC). In contrast to existing feedback controllers, our approach poses the nonconvex constrained nonlinear optimization problem as the tracking of a desired negative mode insertion gradient on the configuration space of a Lie group. This results in a closed-form feedback control law even with input saturation and thus is well suited for online application. In extending SAC to Lie groups, the associated mode insertion gradient is derived and the switching time optimization on Lie groups is studied. We demonstrate the efficacy and scalability of our approach in the 2D kinematic car on $\bm{SE(2)}$ and the 3D quadrotor on $\bm{SE(3)}$. {We also implement iLQG on a quadrator model and compare to SAC, demonstrating that SAC is both faster to compute and has a larger basin of attraction.}

\textit{Index-Terms}--- sequential action control; Lie group; online feedback control; input saturation
\end{abstract}

\IEEEpeerreviewmaketitle

\section{Introduction}

Dynamic robotic tasks require suitable feedback controllers to simultaneously manage dynamics, kinematics, input saturation etc.. In the last thirty years, a number of techniques to design stabilizing control laws have been developed in robotics and control community, such as feedback linearization and backstepping \cite{khalil1996nonlinear,kokolovie1992joy}, linearzing the dynamics around a nominal trajectory to implement the linear quadratic regulator (LQR) or linear model predicative control (LMPC) \cite{anderson2007optimal,bemporad2002model}, dynamic programming \cite{bertsekas1999nonlinear} and nonlinear model predicative control (NMPC) \cite{allgower2004nonlinear,mayne1966second,todorov2005generalized,tassa2014control}, or feedback motion planning via semidefinite programming to precompute a library of control policies filling up the space of possible initial states \cite{jarvis2003some,tedrake2010lqr,majumdar2014convex}. However, feedback linearization and backstepping are generally inapplicable when actuation limits exist while LQR and LMPC only work in a small neighbourhood of the nominal trajectory. Dynamic programming and NMPC are usually computationally expensive and feedback motion planning via semidefinite programming may need lots of time and memory to compute and store the library of control policies if the state space is high-dimensional and noncompact.\par 

Generally speaking, it is preferable to model robot locomotion systems, e.g., vehicular systems, on Lie groups rather than in coordinates since there are no singularity problems. Despite that numerous feedback controllers for systems evolving on Lie groups have been developed \cite{lee2010geometric,kobilarov2014nonlinear,sreenath2013dynamics,huang2014convex,wu2015safety}, all of them remain in the category of aforementioned techniques. Furthermore, feedback controllers on Lie groups resulted from existing control techniques tend to be much more complicated than their coordinate counterparts and as a result, might not be available for time-critical tasks in the presence of actuation limits. Unfortunately, nearly all robot locomotion tasks are required to manage input saturation while being executed in real time. Hence the application of these techniques \cite{lee2010geometric,kobilarov2014nonlinear,sreenath2013dynamics,huang2014convex,wu2015safety} is severely restricted in practice.\par

In this paper, we address the shortcomings of existing control techniques by proposing a novel approach to designing feedback controllers for systems evolving on Lie groups. Our approach is inspired by the recently developed \textit{Sequential Action Control} (SAC) \cite{ansari2016sac}. In contrast to existing control techniques, SAC improves the long-time performance of the controller by tracking a desired negative \textit{mode insertion gradient} and the resulting control action has a \textit{closed-form solution} whose calculation is almost as inexpensive as simulation. More importantly, control actions in SAC can be efficiently saturated and hence it is well suited for robotic tasks in real time. SAC has solved various benchmark problems, such as cart-pendulum, pendubot, acrobot, spring-loaded inverted
pendulum etc. in \cite{ansari2016sac,murphey-ral2016,ansari2015adhs,wilson2015iros}. Nevertheless, up till now, the application of SAC remains limited to nonlinear systems in coordinates and the purpose of this paper is to extend SAC to Lie groups so that we may synthesize control techniques directly on the Lie groups that are commonly used in robotics.
\subsection{Relationship to Existing Control Techniques}\label{subsection::relationship_existing}
{The major difference of SAC from other optimization-based nonliner control techniques \cite{allgower2004nonlinear,tassa2014control,bertsekas1999nonlinear,anderson2007optimal,bemporad2002model,jarvis2003some,tedrake2010lqr,majumdar2014convex,mayne1966second} is that expensive iterative computations of gradients and Hessians are avoided. Control inputs in SAC are computed and saturated based on the mode insertion gradient (\cref{eq::mode_ins}), whose argument is the time duration $\lambda$ instead of control inputs $u(t)$, and which can be exactly evaluated (not approximated) by \cref{eq::mode_ins_vector}. Moreover, most gradient-based methods such as DDP or iLQG attempt to update control inputs over the whole control horizon,  whereas SAC only updates the next control input at a specified future time $\tau$ for a duration $\lambda$.}\par
Although SAC performs well for various nonlinear systems both in coordinates and on Lie groups as these shown in \cite{ansari2016sac,murphey-ral2016,ansari2015adhs,wilson2015iros} and \cref{section::lie_sac}, we are not intending to draw an conclusion that SAC is superior to existing control techniques \cite{khalil1996nonlinear,kokolovie1992joy,allgower2004nonlinear,bertsekas1999nonlinear,anderson2007optimal,bemporad2002model,jarvis2003some,tedrake2010lqr,majumdar2014convex,mayne1966second}. Instead we are more interested in how our approach and existing control techniques can complement each other that results in a win-win solution. \par
As far as we know, SAC augments existing control techniques in the following ways:
\begin{enumerate}
\item SAC is implemented in the beginning to drive the nonlinear system to a region where feedback linearization, backstepping, LQR, LMPC etc. are qualified for the final stabilization even with input saturation.
\item Control actions computed by SAC are used to initialize optimization-based techniques, e.g., NMPC, for which a reasonable initialization is of great significance.
\item In combination with feedback motion algorithms via semidefnitie programming \cite{jarvis2003some,tedrake2010lqr,majumdar2014convex}, which requires to precompute a library of control policies, SAC serves as an early-stage online feedback motion planner to drive the nonlinear system to the target set of initial states that are covered by the feedback motion algorithms via semidefinite programming with a library of much smaller size so that precomputational time and storage memory are saved.
\end{enumerate}

\subsection{Contribution and Organization}
The main contribution of this paper is extending the concept of sequential action control from coordinates to Lie groups so that feedback controllers can be constructed accordingly for input-saturated robotics systems evolving on Lie groups. A great advantage of our approach over SAC in coordinates is that our approach is not required to choose coordinates so that the expensive chart-switching and the associated sudden jump are avoided. In addition to extending SAC to Lie groups, switching time optimization on Lie groups is studied and the first order derivative, i.e., the mode transition gradient, is derived. 

The remainder of this paper is organized as following: \cref{section::sac} reviews sequential action control in coordinates and its advantages. \cref{section::switch} studies the switching time optimization for systems evolving on Lie groups and derives the associated mode insertion and transition gradients. \cref{section::quadratic} analyzes the logarithm-constructed quadratic functions on Lie groups and presents exact expressions of $\d\exp(\xi)$ and $\d\exp^{-1}(\xi)$ for $SO(3)$, $SE(2)$ and $SE(3)$. \cref{section::lie_sac} demonstrates the efficacy and scalability of our approach on two classic robotic systems evolving on Lie groups and \cref{section::conclusion} concludes the paper.
\section{Sequential Action Control}\label{section::sac}
Sequential Action Control is an online model-based approach to high-quality trajectory planning and tracking for nonlinear systems\cite{ansari2016sac}. Compared with general NMPC methods which minimizes the objective iteratively over a finite horizon while enforcing system dynamics as constraints, SAC computes control actions by tracking a desired negative mode insertion gradient, which is a control methodology common in hybrid systems. Since control actions in SAC are assumed to be applied in an infinitesimal duration $\lambda\rightarrow 0^+$, there exists a closed-form solution to control actions having actuation limits which renders SAC well suited for online application. In this section, we review SAC in coordinates and discuss how to extend SAC to Lie groups.\par
\subsection{An Overview of Sequential Action Control}
Consider the affine nonlinear system 
\begin{equation}\label{eq::affine}
\fb{f}{x(t),u(t);t}=\fb{g}{x(t);t}+\fb{h}{x(t);t}\cdot u(t)
\end{equation}
where the state $x(t)\in \R^n$, control action $u(t)\in \R^m$, $\fb{g}{x(t);t}\in\R^n$ and $\fb{h}{x(t);t}\in \R^{n\times m}$. Let 
$$J_1=\int_{T_0}^{T_N} \fb{L}{g(\tau)}\; d\tau +\fb{\varphi}{g(T_N)}$$ be the objective and
suppose the dynamics $\fb{f}{x(t),u(t);t}$ first switches from the nominal mode $$\fb{f_1}{x(t);t}=\fb{f}{x(t),u_1(t);t}$$ to the new mode $$\fb{f_2}{x(t);t}=\fb{f}{x(t),u_2(t);t}$$ at time $\tau_0$ and then back to $\fb{f_1}{x(t),u_1(t);t}$ again after a duration of $\lambda$, i.e., 
\begin{equation} \label{eq::switch_input}
u(t)=\begin{cases}
u_1(t) & t>\tau_0+{\lambda}\;\text{ or } \; t \leq \tau_0,\\
u_2(t) & \tau_0\leq t < \tau_0+{\lambda}.
\end{cases}
\end{equation}
If the duration $\lambda\rightarrow 0^+$ is infinitesimal, the \textit{mode insertion gradient} $\left.\dfrac{dJ_1}{d\lambda}\right|_{0^+}$ is defined to be
\begin{equation}\label{eq::mode_ins}
 \left.\dfrac{dJ_1}{d\lambda}\right|_{0^+}=\lim_{\lambda\rightarrow 0^+}\dfrac{\Delta J_1}{\lambda }
\end{equation}
where $\Delta J_1$ is the objective change yielded by this switch. The model insertion gradient $\left.\dfrac{dJ_1}{d\lambda}\right|_{0^+}$ can be calculated exactly with \cref{eq::mode_ins_vector} \cite{egerstedt2003optimal}
\begin{equation}\label{eq::mode_ins_vector}
\left.\dfrac{dJ_1}{d\lambda}\right|_{0^+}=\rho(\tau_0)^T\cdot \Big[\fb{f_2}{x(\tau_0);\tau_0}- \fb{f_1}{x(\tau_0);\tau_0} \Big]
\end{equation}
where $\rho(t)$ are costates satisfying
\begin{equation}\label{eq::mode_ins_co}
\dot{\rho}(s)^T=-\rho(s)^T \cdot D_x \fb{f_1}{x(s);s}- D_x L\Big(x(s)\Big)
\end{equation}
\begin{equation*}
\text{subject to: } \rho(T_N)=D_g\varphi \Big(g(T_N)\Big).
\end{equation*}
By continuity, the objective change $\Delta J_1$ can be approximated with
$$\Delta J_1\approx \left.\dfrac{dJ_1}{d\lambda}\right|_{0^+}\cdot \lambda$$
if the duration $\lambda$ is small albeit not infinitesimal. So for sufficiently small $\lambda>0$, $J_1$ is reduced in the switch of \cref{eq::switch_input} as long as the mode insertion gradient $\left.\dfrac{dJ_1}{d\lambda}\right|_{0^+}$ is negative.\par

The idea of SAC is to compute control actions $u^*(t)$ reducing the objective $J_1$ by tracking a desired negative mode insertion gradient $\left.\dfrac{dJ_1}{d\lambda}\right|_{0^+}=\alpha_d$ instead of optimizing $J_1$ iteratively. In tracking the negative mode insertion gradient $\alpha_d$, SAC incorporates feedback and computes control actions that guarantee an optimal/near-optimal reduction of $J_1$, which eventually results in a long-time improvement.\par 
The mode insertion gradient tracking can be formulated as an adjoint optimization problem
\begin{multline}\label{eq::J2}
u_2^*(t)=\arg \min_{u_2(t)}\int_{T_0}^{T_N} \!\!\Bigg[\left(\left.\dfrac{dJ_1}{d\lambda}\right|_{0^+}\!\!-\alpha_d\right)^2 + \\ \|u_2(s)\!-\!u_1(s)\|_R^2 \Bigg]\!ds
\end{multline}
where $R\in \R^{m\times m}$ is symmetric positive-definite. For affine system \cref{eq::affine}, it has been proved that \cref{eq::J2} has a closed-form solution
\begin{equation}\label{eq::u2_optimal}
u_2^*(t)=u_1+(\Lambda+R)^{-1}\fb{h}{x(t);t}^T\rho(t)\alpha_d
\end{equation}
where $\Lambda:= \fb{h}{x(t);t}^T \rho(t)\rho(t)^T \fb{h}{x(t);t}$ \cite{ansari2016sac}.

As shown in \cref{algorithm::SAC}, control actions can be saturated in SAC and this is discussed in detail in \cref{subsection::saturation}. For these not discussed, such as determining mode insertion time $\tau_0$ and control duration $\lambda$, interested readers can refer to \cite{ansari2016sac} for a complete introduction.\par

{\begin{algorithm}[t]
\caption{Sequential Action Control}\label{algorithm::SAC}
\begin{algorithmic}[1]
\State Simulate states and costates $(x(t),\,\rho(t))$ with nominal control dynamics $f_1(t)$
\State Compute $u_2^*(t)$ from $(x(t),\,\rho(t))$ by
$$u_2^*(t)=u_1+(\Lambda+R)^{-1}\fb{h}{x(t);t}^T\rho(t)\alpha_d $$
\State Modify $u_2^*(t)$ to satisfy the control bounds $u_{\min}$ and $u_{\max}$
\State Determine the mode insertion time $\tau_0$ to maximize the negative mode insertion gradient
\State Determine the duration $\lambda$ to apply $u_2^*(t)$
\end{algorithmic}
\end{algorithm}
}
\subsection{Saturating Control Actions}\label{subsection::saturation}
If actuation limits exist, the adjoint optimization problem \cref{eq::J2} can be reformulated by quadratic programming with 
$$u_{\min}(t)\leq u_2^*(t) \leq u_{\max}(t) $$
imposed as constraints. This is not preferable considering the high computational expense. Instead, if $R\in \R^{m\times m}$ is diagonal and let ${u'}_2^{*}(t)\!=\!u_1+(\Lambda+R)^{-1}\fb{h}{x(t);t}^T\rho(t)\alpha_d $, then $u_2^*(t)$ can be saturated elementwise by
\begin{equation}\label{eq::u2_solution}
[u_2^*(t)]_i
=\begin{cases}
[u_{\min}]_i & [{u'}_2^*(t)]_i< [u_{\min}]_i,\\
[{u'}_2^*(t)]_i & [u_{\min}]_i\leq[{u'}_2^*(t)]_i\leq [u_{\max}]_i,\\
[u_{\max}]_i & [{u'}_2^*(t)]_i> [u_{\max}]_i.
\end{cases}
\end{equation}
Despite that $u_2^*(t)$ determined by \cref{eq::u2_solution} may not be optimal to \cref{eq::u2_optimal}, it works very well in practice and can guarantee a reduction of $J_1$ if $u_{\min}<u_1(t)<u_{\max}$. Most importantly, the computational expense introduced by \cref{eq::u2_solution} is almost negligible in contrast to quadratic programming.

\subsection{Extension of Sequential Action Control to Lie Groups}
As we can see from the overview above, to extend sequential action control to Lie groups, the first problem is calculating the mode insertion gradient on Lie groups. Unfortunately, \cref{eq::mode_ins_vector,eq::mode_ins_co} may no longer hold for nonabelian Lie groups, meaning that the mode insertion gradient needs to be rederived. We study this problem in \cref{section::switch}, which is the fundamental contribution of this paper. Another problem, though not so obvious, is choosing appropriate objective functions. Though quadratic functions can be constructed easily on Lie groups through the logarithm map, it is difficult to evaluate their derivatives if there are no exact expressions of the trivialized tangent inverse of the exponential map $\d\exp^{-1}(\xi)$, which is discussed in detail in \cref{section::quadratic}.





\section{Mode Insertion and Transition Gradients for Systems Evolving on Lie groups}\label{section::switch}
In this section we study the switching time optimization for systems evolving on Lie groups and derive the associated mode insertion and transition gradients. Readers are assumed to have a basic knowledge of Lie group theory, which can be found in a variety of textbooks, e.g., \cite{hall2003lie}.\par
\subsection{Review of Lie Group Dynamics and Linearization}\label{subsection::dyn_lin}
To be fully-illustrated, we first define the following operator that are frequently used in this paper. 
\begin{definition}
Given a differentiable map $f:G\longrightarrow \R^n$ on a Lie group $G$, then $\grad{g}{f}|_g: \mathfrak{g}\longrightarrow \R^n$ is defined to be the linear map such that
\begin{equation}
\grad{g}{f}|_g\cdot \eta=\left.\frac{d}{d s}\fb{f}{g\cdot\exp(s\cdot\eta)}\right|_{s=0}
\end{equation}
for all $\eta\in \mathfrak{g}$.
\end{definition}
If $G$ is abelian, then $\grad{g}{f}|_g$ is just the first-order derivative in multi-variable calculus. For brevity, the dependence $|_g$ is dropped and $\grad{g}{f}|_g$ is simply written as $\grad{g}{f}$.\par

Let $G$ be the Lie group and $\mathfrak{g}$ the associated Lie algebra, the continuous system evolving on $G$ is defined to be 
\begin{subequations}\label{eq::cont_dyn}
\begin{equation}
\dot{g}(t)=g(t){\xi(t)}.
\end{equation}
\begin{equation} \label{eq::cont_dyn_1}
{\xi}(t)=\fb{f}{g(t),u(t);t}
\end{equation}
\end{subequations}
where $f:G\times \mathcal{U}\times \R^+\longrightarrow \mathfrak{g}$. In the remainder of the paper, we may sometimes write $\fb{f}{g(t),u(t);t}$ as $f(t)$ for simplicity.\par
If $\fb{f}{g(t),u(t);t}$ is $C^1$, it is possible to linearize the corresponding system, which is as follows \cite{saccon2013optimal}.\par
\begin{prop}\label{prop::linear}
If $f:G\times \mathcal{U}\times \R^+\longrightarrow \mathfrak{g}$ is differentiable, then a linearization of continuous system \cref{eq::cont_dyn} is 
\begin{equation}\label{eq::lin}
\dot{\eta}(t)=\Big(\grad{g}{f}(t)-\ad_{f(t)} \Big)\cdot \eta(t)+\grad{u}{f}(t)\cdot \nu(t)
\end{equation}
where $\eta(t)=g(t)^{-1}\delta g(t)\in \mathfrak{g}$ and $\nu(t)\in \mathcal{U}$ is the variation of inputs $u(t)$.
\end{prop}
For more details about Lie group dynamics and linearization, interested readers can refer to \cite{marsden1999introduction,marsden1992lectures,bullo2005geometric,fan2015structured}.\par
\subsection{Mode Transition Gradient}
\begin{definition}
If the inputs $u(t)\in \mathcal{U}$ in \cref{eq::cont_dyn} are transition times 
$$\mathcal{T}=\begin{bmatrix}
T_1 & T_2 & \cdots & T_{N-1}
\end{bmatrix}^T,$$
such that
\begin{multline}\label{eq::input}
\fb{f}{g(t),\mathcal{T};t}=\Big[1(t-T_0)-1(t-T_1^-)\Big]\fb{f_1}{g(t);t}+\\
\Big[1(t-T_1^+)-1(t-T_2^-)\Big]\fb{f_2}{g(t);t}
+\cdots\cdots +\\
\Big[1(t-T_{N-1}^+)-1(t-T_N^-)\Big]\fb{f_N}{g(t);t}
\end{multline}
where $T_0\leq T_1\leq\cdots\leq T_{N-1}\leq T_N$ and $1(t):\R\longrightarrow \{0,\,1\}$ is the unit step function, then \cref{eq::cont_dyn} is a \textit{switched system on Lie group} $G$ and $\fb{f_i}{g(t);t}:G\times\R^+\longrightarrow \mathfrak{g}$ are \textit{mode functions}. \par
\end{definition}
The optimization problem to determine transition times for switched system is called \textit{switching time optimization}. The first- and second-order derivatives of switching time optimization can be evaluated exactly for switched system in coordinates \cite{egerstedt2006transition,caldwell2011switching,egerstedt2003optimal}. However, to our knowledge, no work has been done for switching time optimization on Lie groups. Though switching time optimization is not the main focus of this paper, the \textit{mode transition gradient}, i.e. the first-order derivative of the switching time optimization, is closely related with the mode insertion gradient that SAC tracks in computing control actions $u_2^*(t)$. Here we derive the mode transition gradient for systems evolving on Lie groups that is used to derive the modal insertion gradient in \cref{subsection::mig}. 

\begin{lemma} Suppose $J(\mathcal{T})$ is an objective function of the switched system evolving on Lie group $G$ defined by \cref{eq::cont_dyn,eq::input}
\begin{equation}\label{eq::cost}
J(\mathcal{T})=\int_{T_0}^{T_N} \fb{L}{g(\tau)}\; d\tau +\fb{\varphi}{g(T_N)}.
\end{equation}
Provided each $f_i(\cdot;\cdot)$ in $f$ is $C^1$, then the mode transition gradient $\left.\dfrac{dJ}{d\T}\right|_\mathcal{T}$ of $J(\T)$ is
$$\left.\dfrac{dJ}{d\T}\right|_\T
=\left.\begin{bmatrix}
\left.\dfrac{dJ}{dT_1}\right. & \left.\dfrac{dJ}{dT_2}\right. & \cdots & \left.\dfrac{dJ}{dT_{N-1}}\right.
\end{bmatrix}^T\right|_\T 
$$
such that
\begin{equation}\label{eq::lie_trans}
\left.\frac{dJ}{dT_i}\right|_\mathcal{T}= \rho(T_i)^T \Big[\fb{f_i}{g(T_i);T_i}-\fb{f_{i+1}}{g(T_i),T_i}\Big]
\end{equation}
where
 \begin{multline}\label{eq::costate_diff}
\dot{\rho}(s)^T=-\rho(s)^T \cdot\Big[D_g \fb{f}{g(s),\mathcal{T};s}-\\
 \ad_{f(g(s),\T;s)}\Big] - D_g L\Big(g(s)\Big)
\end{multline}
\begin{equation*}
\text{subject to: } \rho(T_N)=D_g\varphi \Big(g(T_N)\Big).
\end{equation*}
\end{lemma}
\begin{proof}
The derivative $\left.\dfrac{dJ}{dT_i}\right|_\T$ of $J(\T)$ w.r.t $T_i$ is
\begin{multline}\label{eq::variation}
\left.\frac{dJ}{dT_i}\right|_\mathcal{T}=\int_{T_0}^{T_N} D_g L\Big(g(\tau)\Big)\cdot D_{T_i}g(\tau)d\tau + \\
D_g\varphi \Big(g(T_N)\Big)\cdot D_{T_i}g(T_N)
\end{multline}
where $D_{T_i}g(t)=g(t)^{-1}\dfrac{d g(t)}{dt}\in \mathfrak{g}$ and, according to \cref{eq::lin}, satisfies
\begin{multline}\label{eq::first_lin}
\frac{\partial}{\partial t}\Big(D_{T_i}g(t)\Big)=\\ \Big(D_g \fb{f}{g(t),\mathcal{T};t}- \ad_{f(g(t),\T;t)}\Big)\cdot D_{T_i}g(t) + \\ D_{T_i} \fb{f}{g(t),\mathcal{T};t}.
\end{multline}
Note that ${{f_i}}(\cdot;\cdot)$ is $C^1$, thus $\fb{D_gf}{g(t),\T;t}$ is well defined. Since that all initial states are fixed, we have 
$$D_{T_i}g(0)=\0.$$
Therefore, the solution to \cref{eq::first_lin} is
\begin{equation}\label{eq::solution}
D_{T_i}g(t)=\int_{T_0}^{T_N} \Phi(t,s)\cdot  D_{T_i} \fb{f}{g(s),\mathcal{T};s} ds,
\end{equation}
where $\Phi(t,s)$ is the state transition matrix of \cref{eq::first_lin}.
Substitute \cref{eq::solution} to \cref{eq::variation} and switch the integration order, $\left.\dfrac{dJ}{dT_i}\right|_\mathcal{T}$ becomes
\begin{equation}\label{eq::first_1}
\left.\dfrac{dJ}{dT_i}\right|_\mathcal{T}=\int_{T_0}^{T_N} \rho(s)^T\cdot D_{T_i}\fb{f}{g(s),\T;s}ds
\end{equation}
and the costate $\rho(s)^T$ is
\begin{multline}\label{eq::costate}
\rho(s)^T=\int_{s}^{T_N} D_g L\Big(g(\tau)\Big) \Phi(\tau,s) d\tau + \\ D_g\varphi \Big(g(T_N)\Big)\Phi(T_N,s).
\end{multline}
Taking derivative on both sides of \cref{eq::costate}, we have
\begin{multline}
\dot{\rho}(s)^T=-\rho(s)^T \Big[D_g \fb{f}{g(s),\mathcal{T};s}- \\ \ad_{f(g(s),\T;s)}\Big] -D_g L\Big(g(s)\Big),
\end{multline}
$$ \text{subject to: } \rho(T_N)=D_g\varphi \Big(g(T_N)\Big),$$
In addition, note that
\begin{multline}\label{eq::dtf}
D_{T_i} \fb{f}{g(t),\mathcal{T};t}=\delta(t-T_i^-)\fb{f_i}{g(t);t}-\\ \delta(t-T_i^+)\fb{f_{i+1}}{g(t);t}.
\end{multline}
Integrate $\delta$-functions in \cref{eq::first_1} and the result is \cref{eq::lie_trans},
which completes the proof.
\end{proof}

\subsection{Mode Insertion Gradient}\label{subsection::mig}
We have obtained the mode transition gradient for switching time optimization on Lie groups, with which the mode insertion gradient can be derived just as the following corollary indicates.
\begin{corollary}
Given a system evolving on Lie group $G$ with $f: G\times \mathcal{U}\times \R^+\longrightarrow \mathfrak{g}$ taking the form
\begin{equation}\label{eq::transition}
\fb{f}{g(t),\lambda; t}
=\begin{cases}
\fb{f_1}{g(t);t}, & t<\tau_0 \text{\;\;\;or\;\;\;} t\geq\tau_0+\lambda, \\
\fb{f_2}{g(t);t}, & \tau_0\leq t< \tau_0+\lambda
\end{cases}
\end{equation}
and the duration $\lambda\in\R^+$ as input and an objective function 
\begin{equation}\label{eq::lam_cost}
J(\lambda)=\int_{T_0}^{T_N} \fb{L}{g(\tau)}\; d\tau +\fb{\varphi}{g(T_N)},
\end{equation}
then the mode insertion gradient is
\begin{equation}\label{eq::insert}
\left.\frac{d J}{d\lambda} \right|_{\lambda=0^+} = \rho(\tau_0)^T \cdot \Big[\fb{f_2}{g(\tau_0);\tau_0}- \fb{f_1}{g(\tau_0);\tau_0} \Big]
\end{equation}
where
 \begin{multline}\label{eq::costate_diff}
\dot{\rho}(s)^T=-\rho(s)^T \cdot\Big[D_g \fb{f_1}{g(s);s}-\\
 \ad_{f_1(g(s);s)}\Big] - D_g L\Big(g(s)\Big)
\end{multline}
\begin{equation*}
\text{subject to: } \rho(T_N)=D_g\varphi \Big(g(T_N)\Big).
\end{equation*}
\end{corollary}
\begin{proof}
Let $\tau_1=\tau_0+\lambda$ and the dynamical system defined by \cref{eq::transition} is a switched system with $\T=[\tau_1]$. Apply Lemma 1 and note $ \dfrac{dJ}{d\lambda}=\dfrac{dJ}{d\tau_1}$, we have
\begin{multline}\label{eq::mode}
\frac{dJ}{d\lambda}=\rho(\tau_0+\lambda)^T \cdot\Big[\fb{f_2}{g(\tau_0+\lambda);\tau_0+\lambda}- \\
 \fb{f_1}{g(\tau_0+\lambda);\tau_0+\lambda} \Big].
\end{multline}
Substituting $\lambda=0^+$ into \cref{eq::mode}, the resulting equation is \cref{eq::insert}, which completes the proof.
\end{proof}
It can be shown that \cref{eq::u2_optimal} and \cref{eq::u2_solution} apply to SAC on Lie groups as long as the mode insertion gradient \cref{eq::insert} is given which is crucial in extending SAC to nonlinear systems evolving on Lie groups.


\section{The Quadratic Function on Lie Groups and its Derivative}\label{section::quadratic}
The \textit{quadratic function} in $\R^n$
\begin{equation}\label{eq::quad_Rn}
f(x)=\frac{1}{2}\|x-x_d\|_M^2
\end{equation}
where $\|x\|_M=\sqrt{x^T M x}$ and $M\in\R^{n\times n}$ is symmetric positive definite defines a class of functions frequently used in control theory and application. Their essence as Lyapunov functions and easiness in numerical computation render quadratic objective functions quite popular for both SAC and optimization in $\R^n$. This suggests it may be necessary to reasonably extend quadratic functions to Lie groups.\par
\subsection{Extending Quadratic Functions to Lie Groups}
An immediate and common way in differential geometry to extend the quadratic function \cref{eq::quad_Rn} in $\R^n$ to Lie groups is
\begin{equation}\label{eq::quad_lie}
f(g)=\frac{1}{2} \|\log(g_d^{-1}g)\|_M^2
\end{equation}   
whose derivative is
\begin{equation}\label{eq::quad_grad}
D_g f=\d\exp^{-1}\Big(-\log(g_d^{-1}g)\Big)^T M  \log(g_d^{-1}g),
\end{equation}
where $\d\exp^{-1}(\xi)$ is the trivialized tangent inverse of the exponential map.\par 

If Lie group $G$ has a \textit{bi-invariant pesudo-Riemannian metric}, \cref{eq::quad_lie} is a reasonable extension of quadratic functions to Lie groups in sense that $\log(g_d^{-1}g)$ is the initial velocity of the geodesic $\gamma:\,[0,\,1]\longrightarrow G$ connecting $g_d$ and $g$ under the bi-invariant pesudo-Riemannian metric. Fortunately, there always exists such bi-invariant pesudo-Riemannian metrics for subgroups of $SE(3)$ \cite{zefran1996choice}, some of which actually turn out to be positive-definite, i.e., Riemannian metrics.

Though there are a number of other ways to extend quadratic functions from $\R^n$ to Lie groups, such as these in \cite{kobilarov2014discrete,saccon2013optimal}, none of them have the same geometric meaning as \cref{eq::quad_Rn} in $\R^n$ that captures the bi-invariant pesudo-Riemannian structure. Most importantly, according to numerical tests, when using \cref{eq::quad_lie} as the objective, SAC on Lie groups has a larger region of attraction and converges a lot faster than objectives of any other kinds.
\subsection{Expressions of ${\mathrm{d}\exp(\xi)}$ and ${\mathrm{d}\exp^{-1}(\xi)}$ for Some Common Lie Groups}\label{section::dexps}
The trivialized tangent of the exponential map and its inverse $\d\exp(\xi)$ and $\d\exp^{-1}(\xi)$ are frequently used in the construction and linearization of Lie group integrators \cite{saccon2013optimal,hairer2006geometric}. Since $\|\xi\|$ is usually small in Lie group integrators, it is possible to approximate $\d\exp(\xi)$ and $\d\exp^{-1}(\xi)$ by truncating the series expansions \cite{hairer2006geometric}
\begin{equation*}\label{eq::dexp}
\d\exp(\xi)=\sum_{j=0}^\infty \frac{1}{(j+1)!}\text{ad}_\xi^j,\;\;\d\exp^{-1} (\xi)=\sum_{j=0}^\infty \frac{B_j}{j!}\text{ad}_\xi^j
\end{equation*}
where $B_j$ are Bernoulli numbers. However, when $\|\xi\|$ is no longer small, e.g., in an objective function \cref{eq::quad_lie}, the approximation by truncating series expansions is invalid.\par

Here we give exact expressions (not series expansions) of $\d\exp(\xi)$ and $\d\exp^{-1}(\xi)$ for $SO(3)$, $SE(2)$ and $SE(3)$, all of which are commonly used in robotics.
\subsubsection{$SO(3)$} Elements of Lie algebra $\hat{\omega}\in\mathfrak{so}(3)$ is usually associated with $\R^3$ through the hat operator $\wedge:\R^3\longrightarrow \mathfrak{so}(3)$ 
$$\begin{bmatrix}
\omega_1\\
\omega_2\\
\omega_3
\end{bmatrix}^\wedge=\begin{bmatrix}
0 & -\omega_3 &\omega_2\\
\omega_3 & 0 &-\omega_1\\
-\omega_2 & \omega_1 & 0
\end{bmatrix}, $$
then $\d\exp(\hat{\omega})$ and $\d\exp^{-1}(\hat{\omega})$ are \cite{celledoni2003lie}
\begin{subequations}\label{eq::dexps_so3}
\begin{equation}\label{eq::dexp_so3}
\d\exp(\hat{\omega})=\I+\dfrac{1-\cos\|\omega\|}{\|\omega\|^2}\hat{\omega}+\dfrac{\|\omega\|-\sin\|\omega\|}{\|\omega\|^3}\hat{\omega}^2,
\end{equation}
\begin{equation}
\d\exp^{-1}(\hat{\omega})\!=\I-\!\dfrac{1}{2}\hat{\omega}+\dfrac{\frac{1}{2}\|\omega\|\sin\|\omega\|+\cos\|\omega\|-1}{\|\omega\|^2(\cos\|\omega\|-1)}\hat{\omega}^2.
\end{equation}
\end{subequations}
\subsubsection{$SE(2)$} We represent $\xi\in\mathfrak{se}(2)$ in terms of coordinate $\begin{bmatrix}
\omega\\u\\v
\end{bmatrix}\in\R^3$ such that
$
\xi=\begin{bmatrix}
0  &  -\omega & u\\
\omega & 0 & v\\
0 & 0 &0
\end{bmatrix}.
$ The exact expressions of $\d\exp(\hat{\omega})$ and $\d\exp^{-1}(\hat{\omega})$ are
\begin{subequations}\label{eq::dexps_se2}
\begin{equation}
\d\exp(\xi)
=\begin{bmatrix}
1 & 0 & 0\\
\#1 & \dfrac{\sin\omega}{\omega} & -\dfrac{1-\cos\omega}{\omega}\\
\#2 & \dfrac{1-\cos\omega}{\omega} & \dfrac{sin\omega}{\omega}
\end{bmatrix}
\end{equation}
\begin{equation}
\d\exp^{-1}(\xi)
={\begin{bmatrix}
1 & 0 & 0\\
\#3 & \dfrac{\omega}{2}\dfrac{\sin\omega}{1-\cos\omega} & \dfrac{\omega}{2}\\
\#4 & -\dfrac{\omega}{2} & \dfrac{\omega}{2}\dfrac{\sin\omega}{1-\cos\omega}
\end{bmatrix}}
\end{equation}
\end{subequations}
where
\begin{equation*}
\begin{aligned}
&\#1=\dfrac{u(\omega-\sin\omega)+v(1-\cos\omega)}{\omega^2} ,\\
&\#2=\dfrac{v(\omega-\sin\omega)-u(1-\cos\omega)}{\omega^2}, \\
& \#3=-\dfrac{v}{2}+\dfrac{u}{2}\dfrac{\omega\sin\omega+2\cos\omega-2}{\omega\cos\omega-\omega},\\
&\#4=\dfrac{u}{2}+\dfrac{v}{2}\dfrac{\omega\sin\omega+2\cos\omega-2}{\omega\cos\omega-\omega}.
\end{aligned}
\end{equation*}
\subsubsection{$SE(3)$} It is general to identify $\xi\in\mathfrak{se}(3)$ with the body-fixed velocity $\begin{bmatrix}
\omega\\
v
\end{bmatrix}\in\R^3$ through
$\xi=\begin{bmatrix}
\hat{\omega} & v\\
\0 & 0
\end{bmatrix}
$ where $\omega$, $v\in \R^3$ are respectively the body-fixed angular and linear velocities, and we have
\begin{subequations}\label{eq::dexps_se3}
\begin{equation}\label{eq::dexp_se3_1}
\d\exp(\xi)=\begin{bmatrix}
\# 1 & \mathbf{O}_{3\times 3}\\
\# 2 & \#1
\end{bmatrix},
\end{equation}
\begin{equation}\label{eq::dexp_se3_2}
\d\exp^{-1}(\xi)=\begin{bmatrix}
\# 3 & \mathbf{O}_{3\times 3}\\
\# 4 & \#3
\end{bmatrix}
\end{equation}
\end{subequations}
where 
\begin{equation*}
\#1=\I+\dfrac{1-\cos\|\omega\|}{\|\omega\|^2}\hat{\omega}+\dfrac{\|\omega\|-\sin\|\omega\|}{\|\omega\|^3}\hat{\omega}^2,
\end{equation*}
\begin{multline*}
\#2=\dfrac{2-2\cos\|\omega\|-\frac{1}{2}\|\omega\|\sin\|\omega\|}{\|\omega\|^2}\hat{v}+\\
\dfrac{\|\omega\|-\sin\|\omega\|}{\|\omega\|^3}(\hat{\omega}\hat{v}+\hat{v}\hat{\omega})+\\
\dfrac{1-\cos\|\omega\|-\frac{1}{2}\|\omega\|\sin\|\omega\|}{\|\omega\|^4}(\hat{\omega}^2\hat{v}+\hat{\omega}\hat{v}\hat{\omega}+\hat{v}\hat{\omega}^2)+\\
\dfrac{\|\omega\|-\frac{3}{2}\sin\|\omega\|+\frac{1}{2}\|\omega\|\cos\|\omega\|}{\|\omega\|^5}(\hat{\omega}^2\hat{v}\hat{\omega}+\hat{\omega}\hat{v}\hat{\omega}^2)
\end{multline*}
\begin{equation*}
\#3=\I-\dfrac{1}{2}\hat{\omega}+\dfrac{\frac{1}{2}\|\omega\|\sin\|\omega\|+\cos\|\omega\|-1}{\|\omega\|^2(\cos\|\omega\|-1)}\hat{\omega}^2,
\end{equation*}
\begin{multline*}
\#4=-\dfrac{1}{2}\hat{v}+
\dfrac{\frac{1}{2}\|\omega\|\sin\|\omega\|+\cos\|\omega\|-1}{\|\omega\|^2(\cos\|\omega\|-1)}(\hat{\omega}\hat{v}+\hat{v}\hat{\omega})+\\
\dfrac{\frac{1}{4}\|\omega\|^2+\frac{1}{4}\|\omega\|\sin\|\omega\|+\cos\|\omega\|-1}{\|\omega\|^4(\cos\|\omega\|-1)}(\hat{\omega}^2\hat{v}\hat{\omega}+\hat{\omega}\hat{v}\hat{\omega}^2).
\end{multline*}
\cref{eq::dexps_se3} is derived based on the fact that $\hat{\omega}^3=-\|\omega\|^2\hat{\omega}$, in which \cref{eq::dexp_se3_1} is equivalent to that presented in \cite{selig2004lie}.
\section{Examples}\label{section::lie_sac}
In this section, we implement our approach on the 2D kinematic car and 3D quadrotor respectively evolving on $SE(2)$ and $SE(3)$. The results indicate that our approach has a large region of attraction with input saturation. Furthermore, the computation is much faster than real time meaning that our approach can be implemented online. We also compare the performance of SAC with iLQG for quadrotor in in terms of stability and computational time. {All the tests are run in C++ on a 2.7GHz Intel Core i7 Thinkpad T440p laptop}. \par
\subsection{Example 1: The 2D Kinematic Car}
The configuration space of a kinematic car evolves on $g\in SE(2)$ and the dynamic equation can be described as
\begin{subequations}\label{eq::car_model}
\begin{equation}
\dot{g}=g\begin{bmatrix}
0 & -\omega & v^\parallel\cos(\phi)\\
\omega & 0 & 0\\
0 & 0 &0
\end{bmatrix},
\end{equation}
\begin{equation}
{\omega}=v^\parallel\sin(\phi),\;\dot{v}^\parallel=u_1,\;\dot{\phi}=u_2
\end{equation}
\end{subequations}
where $\omega$ is the angular velocity, $v^\parallel$ is the forward velocity, $\phi$ is the steering angle and $u_1,\,u_2\in \R$ are the control inputs.\par

Here we implement SAC for feedback motion planning of car parking. In our tests, the control bounds are $u_{\min}=\begin{bmatrix}
-4 &-5
\end{bmatrix}^T$ and $u_{\max}=\begin{bmatrix}
4 &5
\end{bmatrix}^T$, and to be more consistent with a real car, the steering angle $\phi$ is constrained to $[-\frac{1}{3}\pi,\,\frac{1}{3}\pi]$. The reference control inputs $u_1(t)$ are unknown in motion planning and thus we assume $u_1(t)\equiv\0$ in SAC. However, in parallel parking, it can be checked that $\fb{h}{g(t);t}^T\!\!\rho(t)\equiv\0$ if the weighting matrix $M$ in \cref{eq::quad_lie} is diagonal and the control action $u_2^*(t)\equiv \0$ by \cref{eq::u2_optimal}. Therefore, the car is stuck at the initial state despite the fact that it remains far away from the target. To help the car jump out of the singularity, we let off-diagonal elements of $M$ vary slightly \textit{by random} at intervals instead of keeping $M$ unchanged and then $u_2^*(t)$ no longer remains $\0$.\footnote{$M$ remains positive definite since off-diagonal elements merely slightly vary around 0.} As shown in \cref{fig::car_parking}, this works well for the parallel parking problem.\footnote{The desired steering angle $\phi_d$ is not specified as is seldom considered in motion planning.} \par 

\begin{figure}[]\label{fig::car_parking}
\centering
\vspace{-0.5em}
\begin{tabular}{cc}
  \subfloat[][]{\includegraphics[trim =6mm 4mm 12mm 6mm,width=0.25\textwidth]{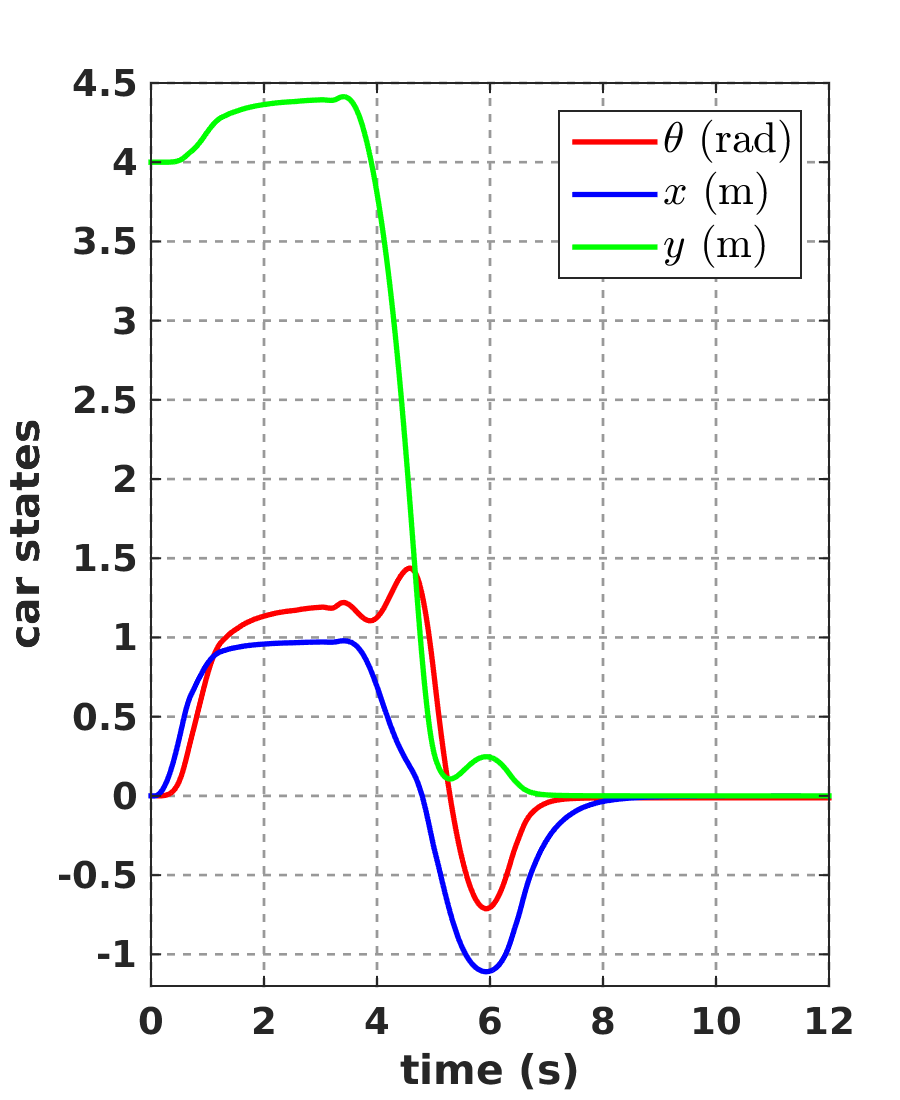}\label{fig::car_parking1}} &
  \hspace{0.1em}
  \subfloat[][]{\includegraphics[trim =4mm 4mm 8mm 6mm,width=0.2\textwidth]{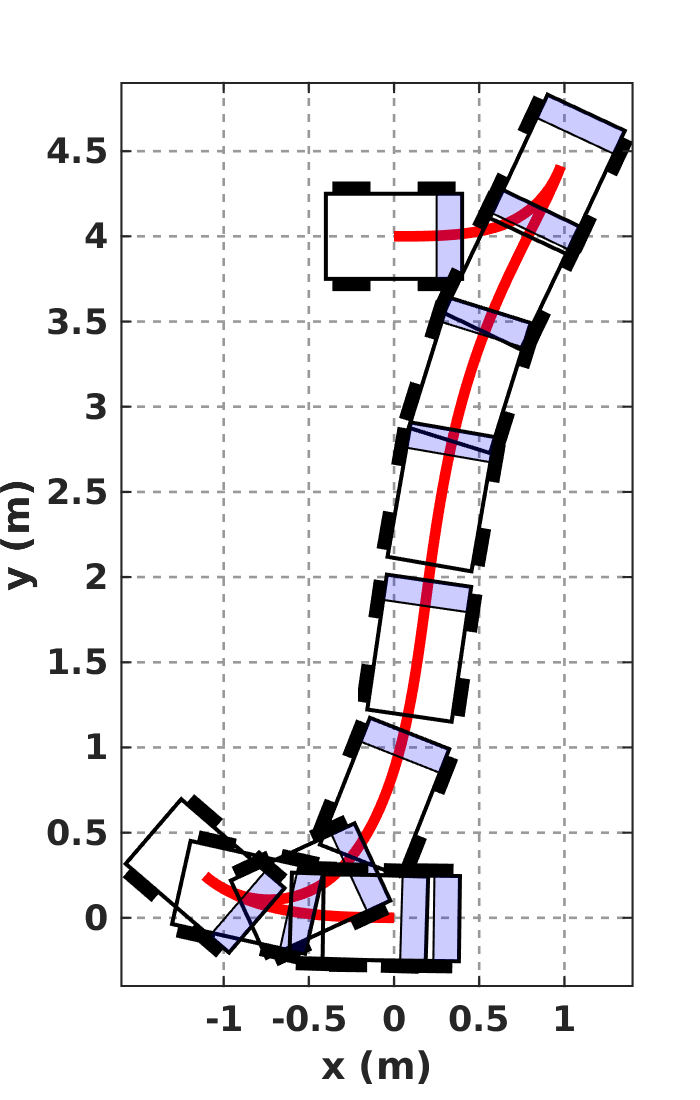}\label{fig::car_parking2}} 
\end{tabular}
\caption{Feedback motion planning for parallel parking. The initial state is $\theta_0=0\,\text{rad}$, $x_0=0\,\text{m}$, $y_0=4\,\text{m}$, $\phi_0=0\,\text{rad}$, $\omega_0=0\,\text{rad}/ \text{s}$, $v_0=0\,\text{m}/ \text{s}$ and the desired state is $\theta_d=0\,\text{rad}$, $x_d=0\,\text{m}$, $y_d=0\,\text{m}$, $\omega_d=0\,\text{rad}/ \text{s}$, $v_d=0\,\text{m}/ \text{s}$. It takes $0.4\,\text{s}$ for SAC to successfully plan a trajectory to the desired configuration. The feedback sampling rate is $100$ Hz.\vspace{-1.5em}}  
\label{fig::car_parking}
\end{figure}

We also test SAC with around 45000 initial states randomly sampled from $\theta_0\in[-\pi,\,\pi]\,\mathrm{rad}$, $x_0,\,y_0\in[-10,\,10]\,\mathrm{m}$. In fact, as stated in \cref{subsection::relationship_existing}, rather than make SAC search feedback motion plans all by itself, we prefer to combine it with feedback motion planning algorithms via semidefinite programming. In all of our tests, SAC successfully drives the car from initial states aforementioned into a region of $|\theta_{T}|\leq 0.15\,\mathrm{rad}$, $\|\mathbf{x}_{T}\|\leq 0.35\,\mathrm{m}$, $|v^\parallel_T|\leq 0.02\,\mathrm{m}/\mathrm{s}$ after a computational time of around $0.8\, \mathrm{s}$, which can be covered with significantly fewer control policies by feedback motion planning algorithms via semidefnite programming \cite{jarvis2003some,tedrake2010lqr,majumdar2014convex}.\par

\begin{figure*}
\centering
\vspace{-1em}
\begin{tabular}{ccc}
  \subfloat[][]{\includegraphics[trim =23mm 2mm 5mm 2mm,width=0.32\textwidth]{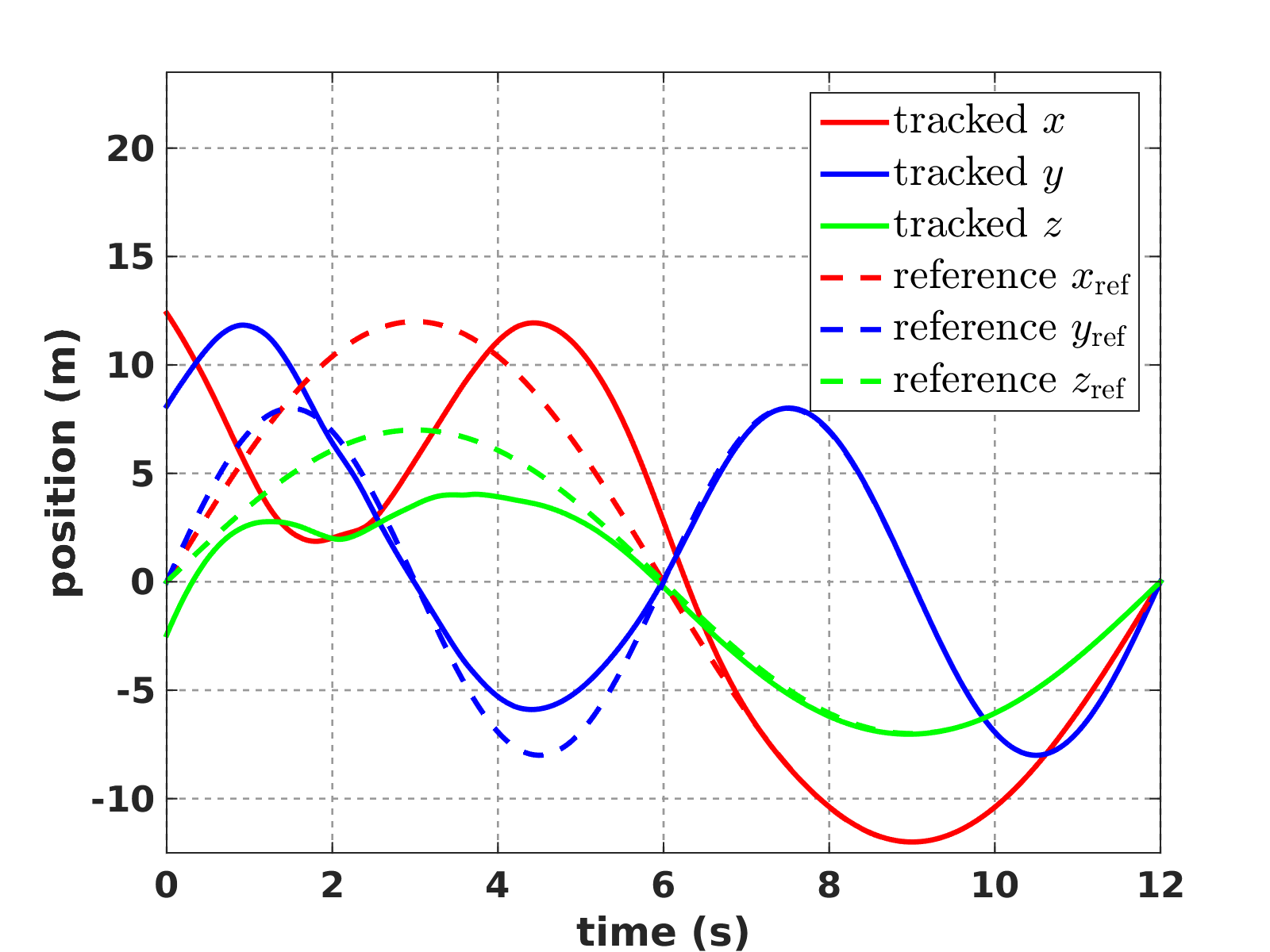}} &
  \hspace{0.5em}
  \subfloat[][]{\includegraphics[trim =23mm 2mm 5mm 2mm,width=0.32\textwidth]{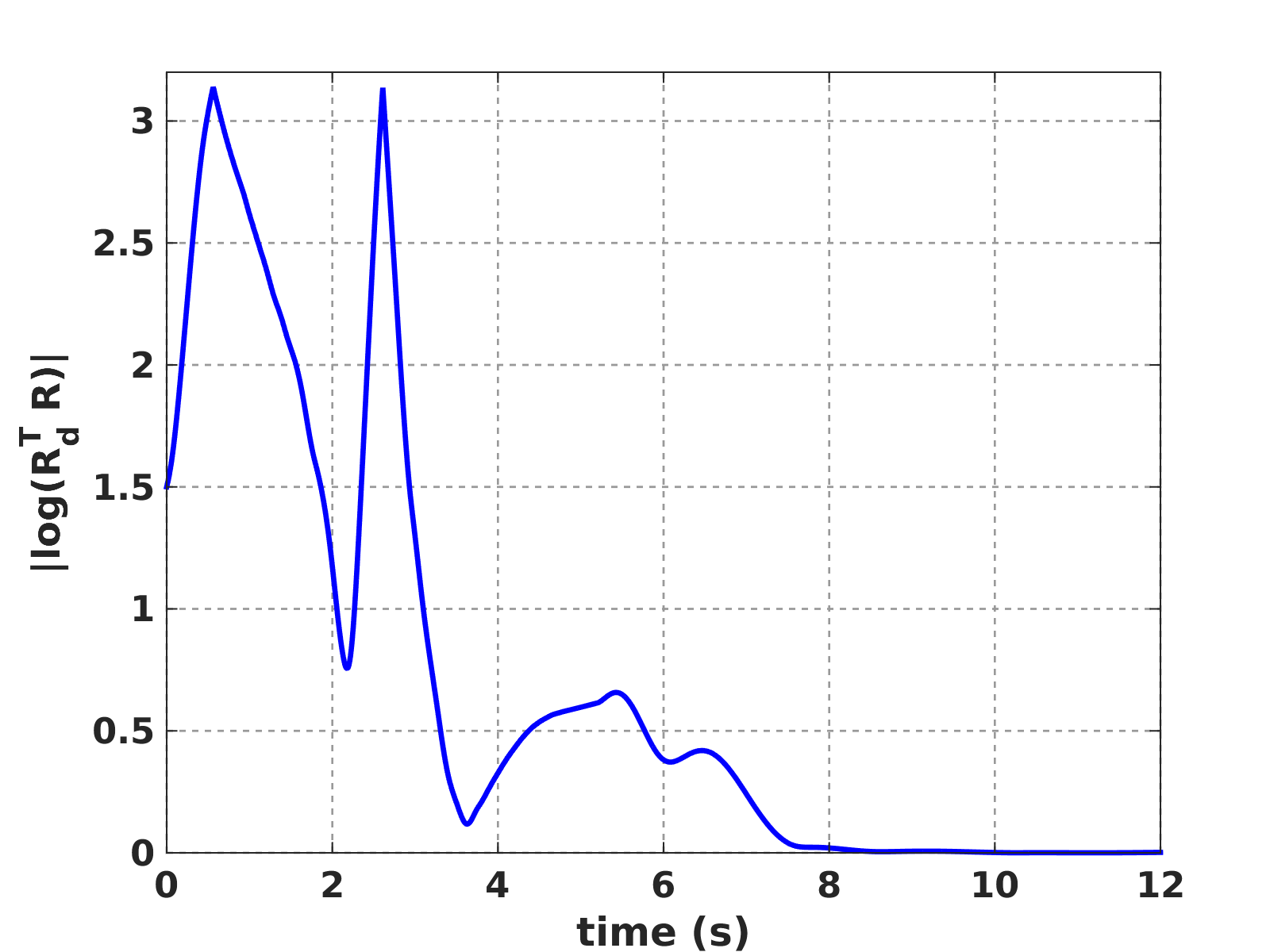}} &
  \subfloat[][]{\includegraphics[trim =20mm 2mm 5mm 2mm,width=0.33\textwidth]{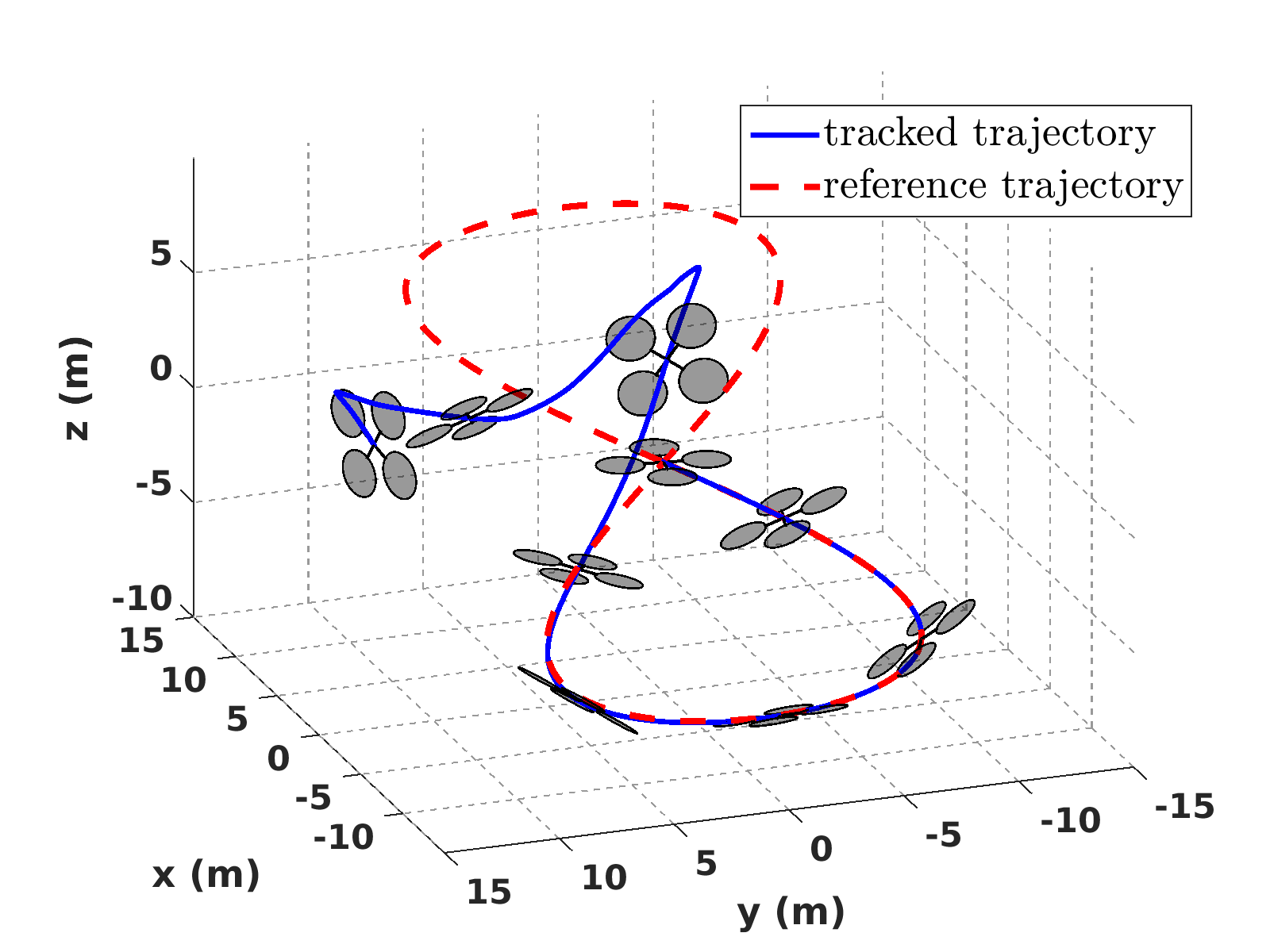}}
\end{tabular}
\caption{Trajectory tracking of a quadrotor with large initial error. The reference initial state is $\alpha=0$, $\beta=0$, $\gamma=0$, $x_{q}=[0\;0\;0]^T\,\text{m}$,  $\omega=[0.94\;\;-0.18\;\;0.25]^T\,\text{rad}/\text{s}$ and $v_q=[6.28\;\;8.38\;\;3.67]^T\,\text{m}/\text{s}$ while the tracking initial state is $\alpha=1.45$, $\beta=-0.92$, $\gamma=-0.70$ $x_{q}=[12.38\;\;8.10\,\,-2.44]^T\,\text{m}$,  $\omega=[-0.56\;0.90\;3.80]^T\,\text{rad}/\text{s}$ and $v_q=[10.39\;4.17\;4.85]^T\,\text{m}/\text{s}$, where $\alpha$, $\beta$, $\gamma$ are respectively yaw, pitch, roll angles. The resulting performance of position tracking is in (a), orientation tracking in (b) and trajectory tracking in (c). The LQR controller takes effect at $t=5.20\text{ s}$. SAC is around 5.3 times faster than real time. The results are based on a feedback sampling rate of $50$ Hz.\vspace{-0.5em}} 
\label{fig::quadrotor} 
\end{figure*}
\begin{figure*}
\centering
\vspace{-1em}
\begin{tabular}{ccc}
  \subfloat[][$t=18\,\mathrm{s}$]{\includegraphics[trim =23mm 2mm 5mm 2mm,width=0.32\textwidth]{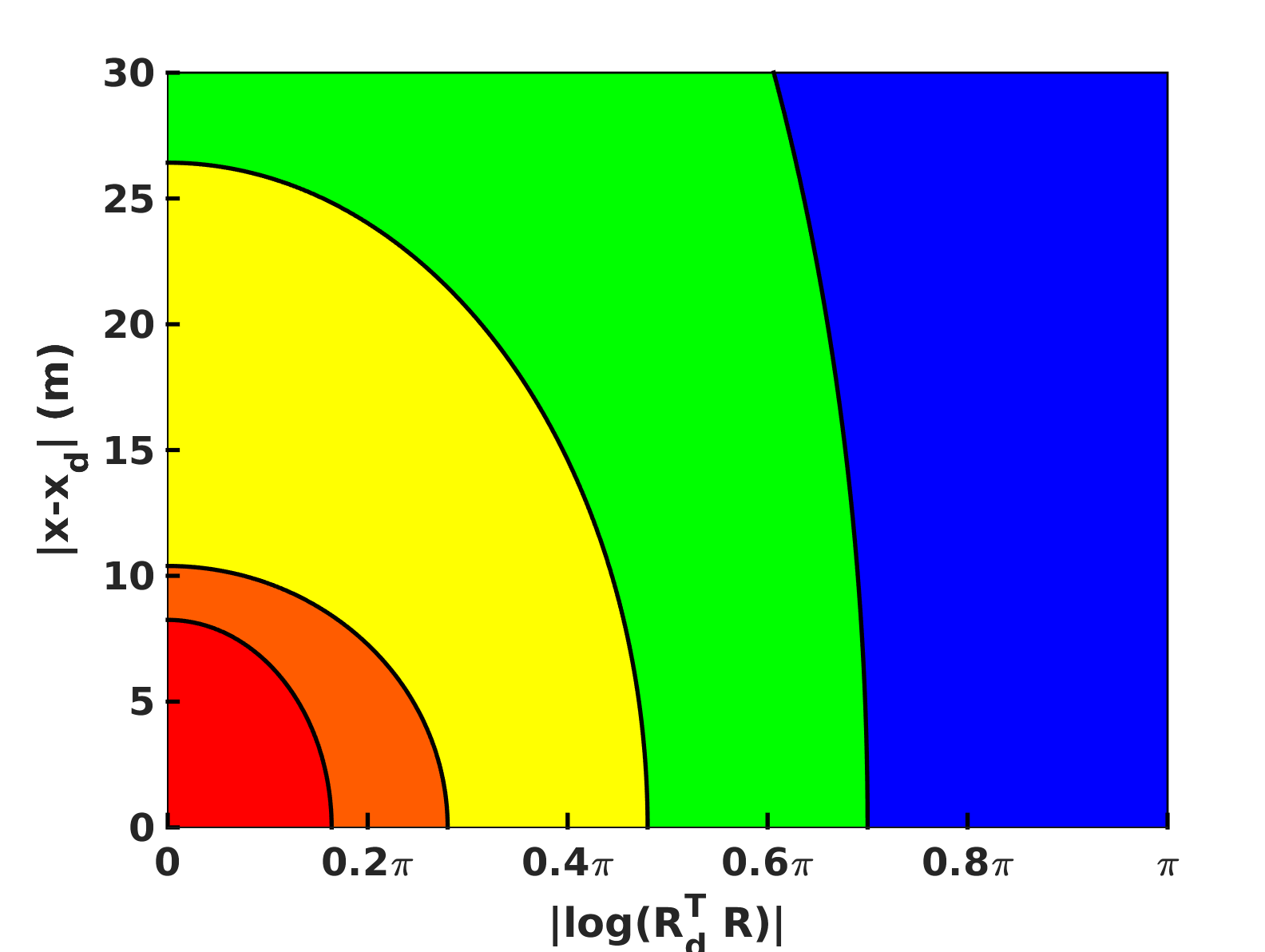}\label{fig::roa_1}} &
  \subfloat[][$t=30\,\mathrm{s}$]{\includegraphics[trim =23mm 2mm 5mm 2mm,width=0.32\textwidth]{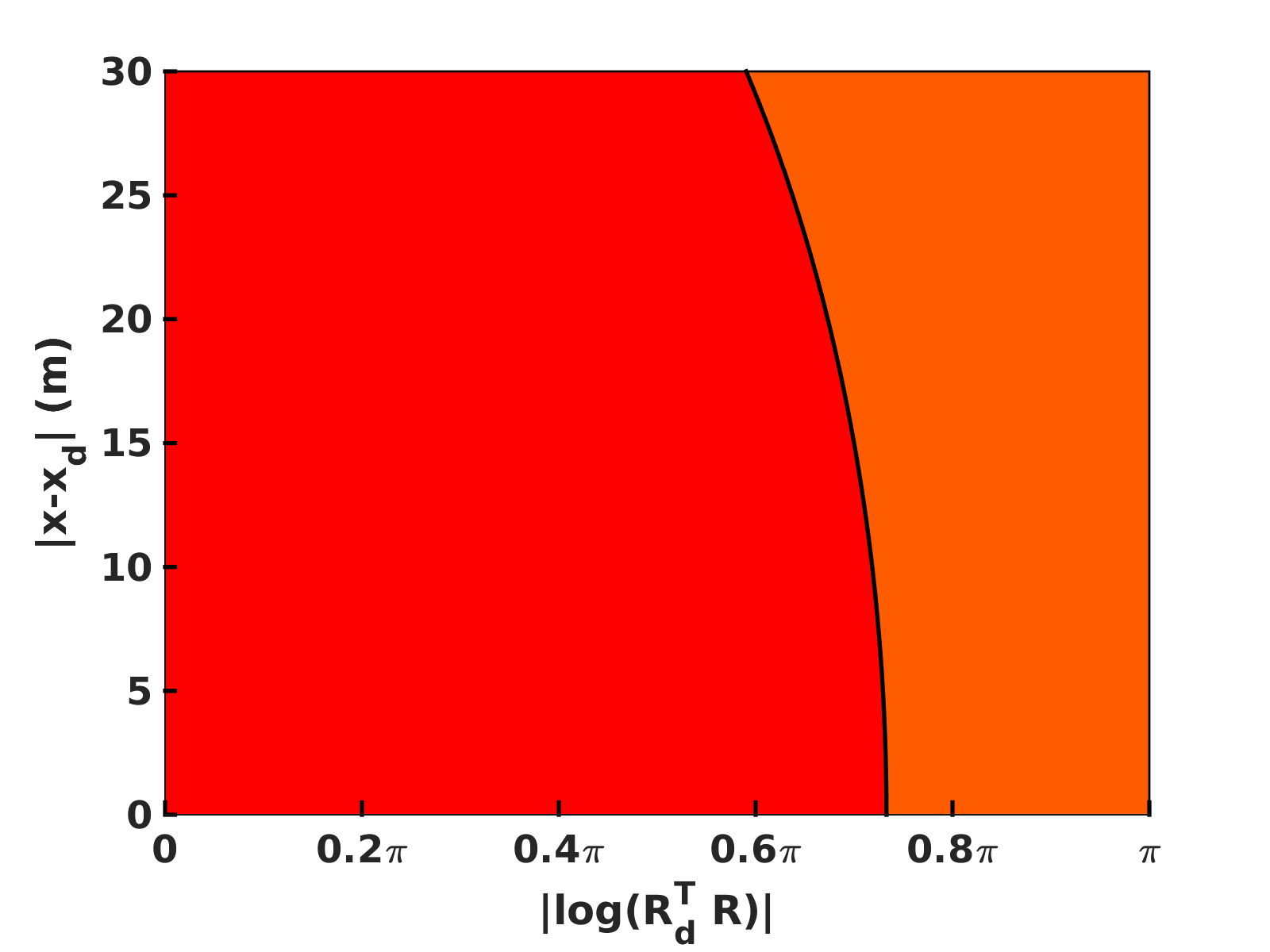}\label{fig::roa_2}}&
    \subfloat[][]{\includegraphics[trim =21mm 0mm 2mm 0mm,width=0.32\textwidth]{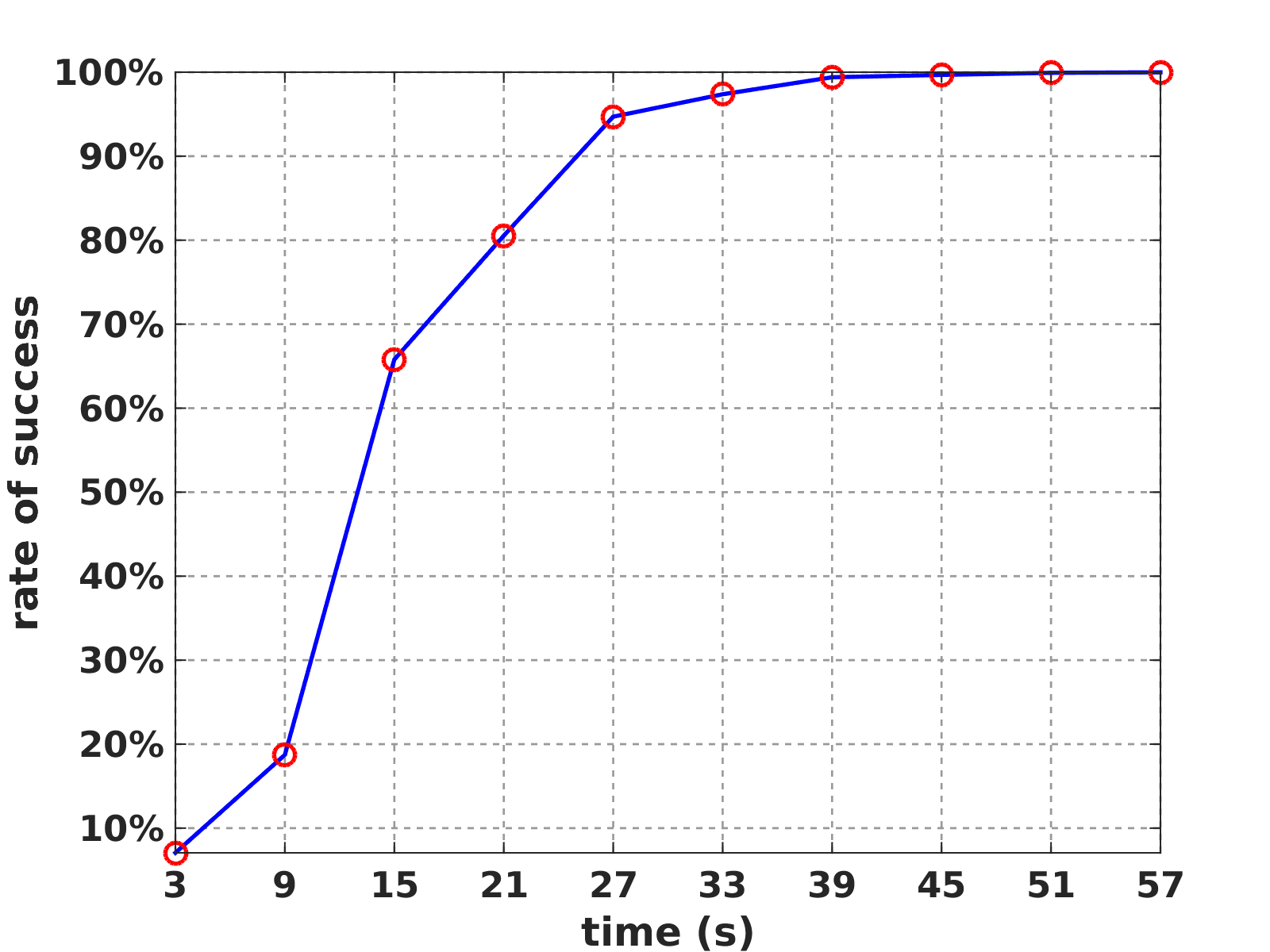}\label{fig::roa_3}} 

\end{tabular}
\caption{Statistics of SAC for quadrotor control with 24000 random initial states. In (a) and (b) different colors denote rates of success converging to the region of attraction of LQR at simulation times of $t=18\,\mathrm{s}$ and $30\,\mathrm{s}$ for different initial errors where red $\geq 95\%$, orange $\geq 90\%$, yellow $\geq 80\%$, green $\geq 70\%$, blue $\geq 60\%$ while (c) is the overall rate of success w.r.t. different simulation times. At $t=63.4\,\mathrm{s}$, SAC successfully drives the quadrotor to the region of attraction of LQR for all of the 24000 initial states. Note herein the region of attraction of LQR refers to that in the presence of actuation limits.\vspace{-0.9em}} 
\label{fig::attraction} 
\end{figure*}
\subsection{Example 2: The 3D Quadrotor}\label{subsection::quadrotor}
The quadrotor is an underactuated system evolving on $h=(R,\,p)\in SE(3)$ and the dynamics can be formulated as
\begin{subequations}\label{eq::quadrotor}
\begin{equation}
\dot{h}=h\begin{bmatrix}
\hat{\omega} & v\\
\0 & 0
\end{bmatrix},
\end{equation}
\begin{equation}
J\dot{\omega}=M+J\omega\times\omega,
\end{equation}
\begin{equation}
\dot{v}=\frac{1}{m}Fe_3-\omega\times v-gR^Te_3
\end{equation}
\end{subequations}
where $\omega,\,v\in\R^3$ are respectively the body-fixed angular and linear velocities, $F$ and $M$ are forces and torques exerted on the quadrotor associated with control inputs $u_i$ ($i=1,\,2,\,3,\,4$) by
\begin{equation}\label{eq::quadrotor_u}
\begin{aligned}
F&=k_t(u_1^2+u_2^2+u_3^3+u_4^2),\\
M&=\begin{bmatrix}
k_tl(u_2^2-u_4^2)\\
k_tl(u_3^2-u_1^2)\\
k_m(u_1^2-u_2^2+u_3^3-u_4^2)
\end{bmatrix},
\end{aligned}
\end{equation}
and $k_t$, $k_m$, $l$ are model parameters. The dynamics (\cref{eq::quadrotor}) used herein are derived on $SE(3)$ and it is different from that on $SO(3)\times\R^3$ due to the distinct Lie group structure\cite{lee2010geometric}.\par

Since in \cref{eq::quadrotor_u} $u_i^2\geq 0$ always holds, values of $M$ and $F$ are no longer arbitrary. Though numbers of papers develop various quadrotor controllers, most of them are either NMPC-based or assume $M$ and $F$ can be any values, which may not be implemented online with input saturation. 

The fact that the quadrotor dynamics is an affine system with $u_i^2\in [0,+\infty)$ means that SAC on Lie groups can be implemented to saturate control actions. In this example, we use the combination of SAC and LQR on a quadrotor with $m=0.6\,\text{kg}$, $J=\text{diag}\{ 0.04,\,0.0375,\,0.0675\}\,\text{kg}\cdot \text{m}^2$, $k_t=0.6$, $k_m=0.15$, $l=0.2\,\text{m}$ and $u_i^2\in[0,\,6]$ for $i=1,\,2,\,3,\,4$. The LQR is employed only when the tracking error is below a threshold of $\|\log(g_d^{-1}g)\|^2+\|\omega-\omega_d\|^2+\|x-x_d\|^2\leq 6^2$ and the control inputs given by LQR satisfy the actuation limits, otherwise SAC is applied. The reference control inputs are obtained by differential flatness \cite{mellinger2011minimum} and the flat outputs are $x(t)=12\sin(\frac{\pi t}{6})\,\mathrm{m}$, $y(t)=8\sin(\frac{\pi t}{3})\,\mathrm{m}$, $z(t)=7\sin(\frac{\pi t}{6})\,\mathrm{m}$ and $\alpha=\frac{1}{2}\sin(\frac{t}{2})\,\mathrm{rad}$. \cref{fig::quadrotor} is a result of trajectory tracking while \cref{fig::attraction} are the statistics of 24000 trials whose initial states are randomly sampled from $\|x-x_d\|\in[0,\,30]\,\mathrm{m}$, $\|\log(R_d^{T}R)\|\in[0,\,\pi]$ and in all of which the combination of SAC and LQR stabilizes the quadrotor to the reference trajectory in a simulation duration of $72\,\mathrm{s}$ with saturated control inputs.\par
\begin{figure}[H]
\centering
\vspace{-0.65em}
\begin{tabular}{c}
  \subfloat[][]{\includegraphics[trim =12.5mm 2mm 5mm 12.5mm,width=0.37\textwidth]{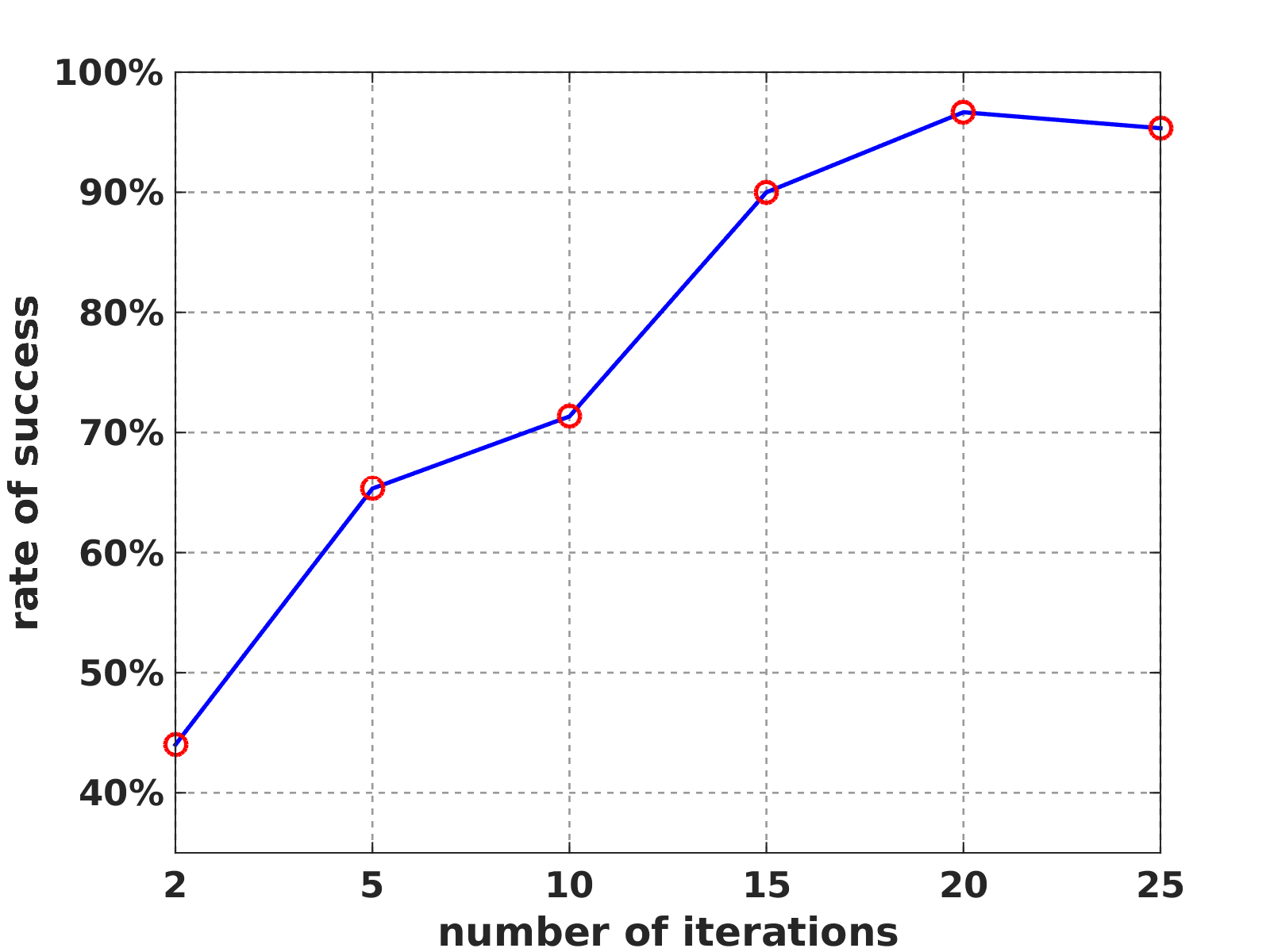}\label{fig::ilqg_1}} \\
  \subfloat[][]{\includegraphics[trim =12.5mm 2mm 5mm 12.5mm,width=0.37\textwidth]{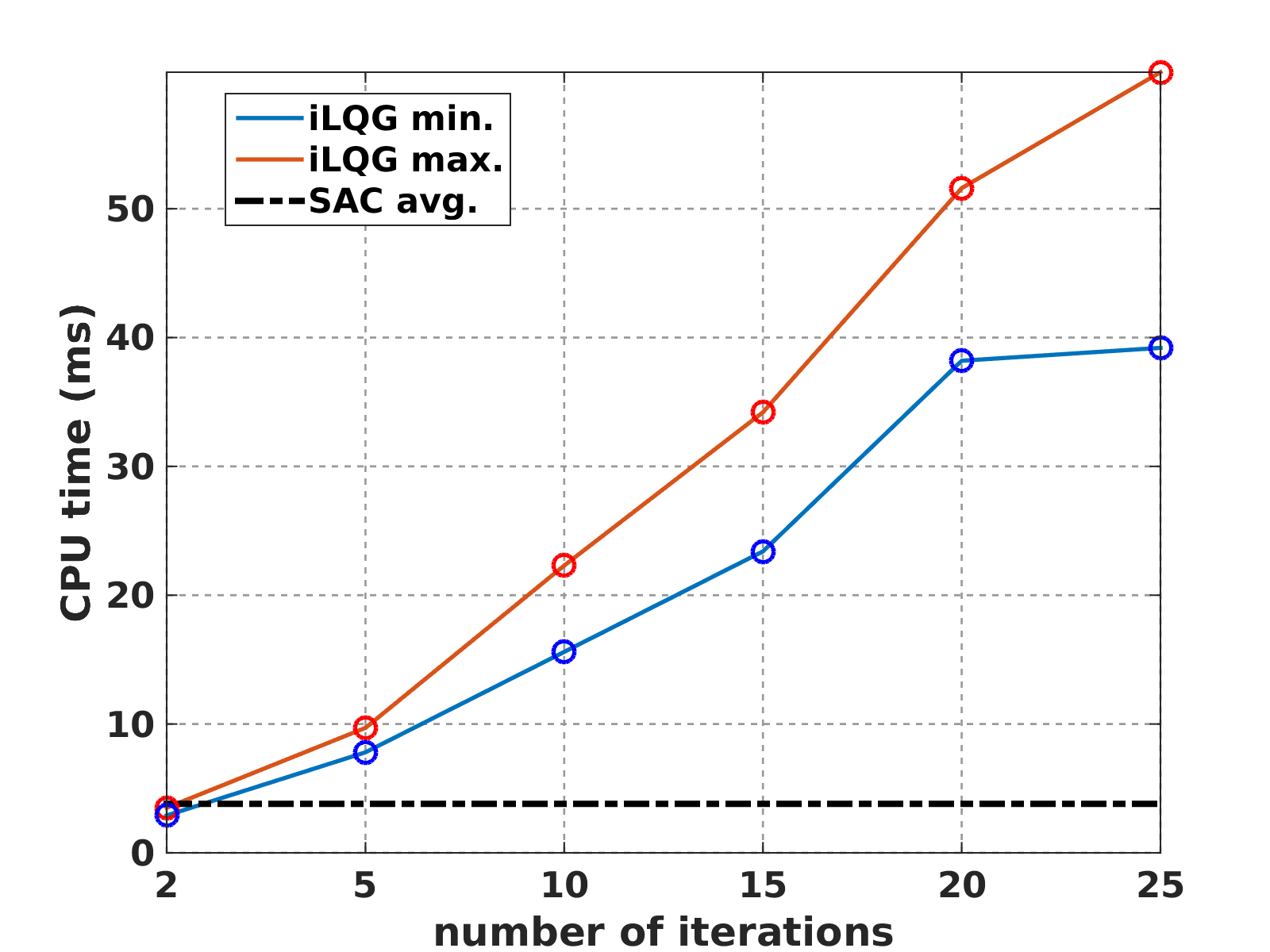}\label{fig::ilqg_2}}
\end{tabular}
\caption{{The performance of iLQG under different number of iterations. The stability is tested with 150 random initial states with $\|x-x_d\|=30$ m and $\|\log(g_d^{-1}g)\|=3$ rad and the rate of success of iLQG is shown in (a). The computational time of iLQG with different numbers of iterations is in (b) where the black dotted line represents the average computational time of SAC. The computational time of SAC is around $3.8\sim 4.1$ ms.}} \label{fig::ilqg}
\vspace{-0.7em}
\end{figure}
{We implement iLQG \cite{todorov2005generalized,tassa2014control} for trajectory tracking for the purpose of comparison. The iLQG method is one of the most efficient NMPC methods. As is shown in \cref{fig::attraction,fig::ilqg}, SAC outperforms iLQG in both basin of attraction and computational efficiency for cases with large initial errors. Though iLQG achieves a relatively satisfactory performance with 20 iterations, it takes $45$ ms for computation and still has failures. As a comparison, SAC takes around $4$ ms and stabilizes quadrator in all of the 24000 trials including these with large initial errors (\cref{fig::attraction}).}\par
\begin{table}[H]
    \begin{center}
    \vspace{-0.5em}
    \begin{tabular}{|c|c|c|c|}
        \hline
        \multirow{2}{*}{} & {\multirow{2}{*}{SAC}} &\multicolumn{2}{c|}{iLQG}\\
        \cline{3-4}
        &  & SNOPT & Projected-Newton \\
        \hline
        \multirow{2}{*}{CPU time} & \multirow{2}{*}{$3.8 \sim 4.1$ ms} &\multirow{2}{*} {$217\sim 236$ ms} &\multirow{2}{*} {$37\sim 51$ ms} \\
        & & & \\
        \hline
    \end{tabular}
    \end{center}
    \caption{{Computation time of SAC and iLQG. In our tests, the number of iterations in iLQG is set to 20 to obtain an accepted performance, whereas no iterations are needed in SAC.}\vspace{-0.8em}}\label{compare1}
\end{table}
{Usually iLQG needs a quadratic optimizer to saturate control inputs. We tested both SNOPT \cite{gill2005snopt} and the Projected-Newton method \cite{tassa2014control, bertsekas1982projected} for quadratic optimization in iLQG. As noted in \cref{compare1}, SAC is much faster than iLQG using both optimizers. We can also find that the performance of iLQG is severely affected by the chosen optimizers while SAC saturates control inputs in closed form and no optimizers are needed.}

\section{Conclusion}\label{section::conclusion}
In this paper, we propose an approach to designing online feedback controllers with input saturation for nonlinear systems evolving on Lie groups by extending sequential action control, 
which demonstrates a large region of attraction for the kinematic car and quadrotor model though formal proofs of global stability are still needed. In addition, the associated mode insertion and transition gradients are derived and exact expressions of $\d\exp(\xi)$ and $\d\exp^{-1}(\xi)$ for some common Lie groups are given.
\bibliographystyle{plainnat}
\bibliography{mybib}
\end{document}